\documentclass{amsart}
\usepackage{amssymb, amsmath, amssymb, enumerate, verbatim}
\usepackage{color, graphicx}
\usepackage{txfonts}
\usepackage{dsfont}
\usepackage{hyperref}
\hypersetup{colorlinks=true,citecolor=black,linkcolor=black,anchorcolor=black,filecolor=black,menucolor=black,urlcolor=black,pdfauthor={V. Kapovitch, N. Li},pdfdisplaydoctitle}
\newtheorem{thm}{Theorem}[section]

\newtheorem*{thm*}{Theorem} 
\newtheorem{cor}[thm]{Corollary}
\newtheorem{lem}[thm]{Lemma}
\newtheorem{defn}[thm]{Definition}
\newtheorem{ex}[thm]{Example}

\newtheorem{quest}[thm]{Question}

\newtheorem{prop}[thm]{Proposition}

\theoremstyle{definition}

\newtheorem*{rems*}{Remarks}
\theoremstyle{remark}

\newcommand{\R}{\mathbb{R}}

\newcommand{\Rg}{\mathcal{R}}
\newcommand{\Mx}{\mathop{\rm Mx}}

\newcommand{\Mxr}{\mathrm {Mx} _r}
\newcommand{\MxR}{\mathrm {Mx}_R}
\newcommand{\Mxrh}{\mathrm {Mx} _\rho}

\DeclareMathOperator{\Ric}{Ric}

\DeclareMathOperator{\dt}{dt}

\DeclareMathOperator{\divr}{div}
\DeclareMathOperator{\Hess}{Hess}

\DeclareMathOperator{\vol}{vol}
\DeclareMathOperator{\Vol}{vol}

\newcommand{\eps}{\varepsilon}

\newcommand{\ene}{\end{equation} }

\newcommand{\ba}{\begin{eqnarray}}
\newcommand{\ea}{\end{eqnarray}}
\newcommand{\ban}{\begin{eqnarray*}}
\newcommand{\ean}{\end{eqnarray*}}

\newcommand{\sfint}{{-}\hspace*{-0.90em}\int}
\def\co{\colon\thinspace}

\def\DD{\mathbb D}

\def\a{\alpha}

\newtheorem{remark}[thm]{Remark}
\newenvironment{rmk}{\begin{remark}\rm}{\end{remark}}
\newtheorem{Fact}[thm]{Fact}

\newtheorem{Nothing}[thm]{$\!\!\!$}

\numberwithin{equation}{section}

\makeindex
\begin{document}
\thanks{The first author was supported in part by a Discovery grant from NSERC.
The second author was supported in part by CUNY PDAC Travel Award}
\abovedisplayskip=6pt plus3pt minus3pt \belowdisplayskip=6pt
plus3pt minus3pt
\title[On dimensions of tangent cones in limit spaces with lower Ricci curvature bounds]{On dimensions of tangent cones in limit spaces with lower Ricci curvature bounds}

\author{ Vitali Kapovitch and Nan Li}

\address{Vitali Kapovitch\\Department of Mathematics\\
University of Toronto\\
Toronto, ON, Canada M5S 2E4}\email{vtk@math.toronto.edu}

\address{Nan Li\\Department of Mathematics\\
CUNY -- New York City College of Technology\\
300 Jay Street \\
Brooklyn, NY 11201\\  USA}
\email{NLi@citytech.cuny.edu}


\begin{abstract}
  We show that if $X$ is a limit of $n$-dimensional Riemannian manifolds with Ricci curvature bounded below and $\gamma$ is a limit geodesic in $X$ then along the interior of $\gamma$ same scale measure metric tangent cones
$T_{\gamma(t)}X$ are H\"older continuous with respect to measured Gromov-Hausdorff topology and have the same  dimension in the sense of Colding-Naber.

\end{abstract}
\setcounter{page}{1} \setcounter{tocdepth}{0}

\maketitle

\section{Introduction}
In this paper we obtain new continuity results for tangent cones along interiors of limit geodesics in Gromov-Hausdorff limits of manifolds with lower Ricci curvature bounds.

Our main technical result is the following

\begin{thm}\label{thm-main} For any $H\in\R$ and $0<\delta<1/3$,
there exist $ r_0(n,\delta,H)$, $\eps(n,\delta, H)>0$ and $0<\alpha(n)<1$ such that the following holds:
\\
Suppose that $(M^n,g)$ is a complete $n$-dimensional Riemannian manifold with $\Ric_M\ge (n-1)H$ and let $\gamma\co [0,1]\to M$ be a unit speed minimizing geodesic.
Then for any $t_1,t_2\in (\delta,1-\delta)$ with $|t_1-t_2|<\eps$ and any $r<r_0$ there exist subsets $\mathcal C_i^r\subset B_r(\gamma(t_i))$   ($i=1,2$) with
\[
\frac{\vol \mathcal C_i^r}{\vol B_r(\gamma(t_i))}\ge 1-|t_1-t_2|^{\alpha(n)}
\]
and a $\left(1+|t_1-t_2|^{\alpha(n)}\right)$-Bilipschitz onto map $f_r\co \mathcal C_1^r\to \mathcal  C_2^r$, that is, $f_r$ is bijective and
\[
  \left|\frac{d(f_r(x),\,f_r(y))}{d(x,\,y)}-1\right|\le |t_1-t_2|^{\alpha(n)}
\]
for any $x,y\in \mathcal C_1^r$ with $x\neq y$.
\end{thm}

Let $d\vol_{i,r}=\frac {d\vol }{\vol B_r(\gamma(t_i))}$ ($i=1,2$) be the renormalized volume measures at $\gamma(t_i)$. It's then  obvious  that under the assumptions of the theorem  we have
\begin{align}
\left(1-C(n,\delta)|t_1-t_2|^{\alpha(n)}\right)\,d\vol_{2,r}\le (f_r)_\#(d\vol_{1,r})\le \left(1+C(n,\delta)|t_1-t_2|^{\alpha(n)}\right)\,d\vol_{2,r}\,
\end{align}
for some universal $C(n,\delta)>0$.

Let $(M_j^n,q_j)\to (X,q)$ where $\Ric_{M_j}\ge (n-1)H$.  By passing to a subsequence we can assume that the renormalized volume measures $\frac{d\vol_{M_j}}{\vol B_1(p_j)}$ on $M_j$ converge to a  measure $d\vol$ on $X$~\cite{CC1}.
For a point $x\in X$ let $\displaystyle (T_xX,o_x)=\lim_{k\to\infty} (r_kX,x)$ be a tangent cone at $x$ corresponding to some $r_k\to\infty$.

Again, up to  passing to a subsequence we can assume that the renormalized measures $\frac{d\vol}{vol B_{1/r_k}(x)}$ converge to a renormalized measure $d\vol_x$ on $T_xX$ (Note that $\vol_x(B_1(o_x))=1)$.

Given $x_1,x_2\in X$ we will call  tangent cones $(T_{x_i}X,o_{i},d\vol_i)$ $i=1,2$ together with the limit measures {\it same scale} if they come from the same rescaling sequence $r_k\to\infty$.

Using precompactenss and a standard Arzela-Ascoli type argument Theorem~\ref{thm-main} easily yields
\begin{cor}\label{cor-main}
For any $H\in\R$ and $0<\delta<1/3$,
there exist  $\eps(n,\delta, H)>0$ and $0<\alpha(n)<1$ such that the following holds:

Let $M_j^n\to X$ where $\Ric_{M_j}\ge (n-1)H$.   Let $\gamma\co [0,1]\to X$ be a unit speed geodesic which is a limit of geodesics in $M_i$. Let $d\vol$ be a renormalized limit volume measure on $X$.

 Then for any $t_1,t_2\in (\delta,1-\delta)$ with $|t_1-t_2|<\eps$  there exist  subsets  $\mathcal C_i$ ($i=1,2$) in the unit ball around the origin $o_i$ in the same scale tangent cones  $(T_{\gamma(t_i)}X, d\vol_i)$ ($i=1,2$) such that
\[
{\vol_i \mathcal C_i}\ge 1-|t_1-t_2|^{\alpha(n)}
\]
and there exists a map $f\co \mathcal C_1\to \mathcal C_2$ satisfying
\begin{enumerate}
  \renewcommand{\labelenumi}{(\roman{enumi})}
  \item $f$ is $(1+|t_1-t_2|^{\alpha(n)})$-Bilipschitz onto;
  \item
      $$(1-|t_1-t_2|^{\alpha(n)})\,d\vol_2\le f_\#(d\vol_1)\le (1+|t_1-t_2|^{\alpha(n)})\,d\vol_2.$$
      In particular, $f_\#(d\vol_1)$ ($f^{-1}_\#(d\vol_2)$)  is absolutely continuous with respect to $\vol_2$ ($\vol_1$).
\end{enumerate}

\end{cor}
In~\cite{CoNa} Colding and Naber show that under the assumptions of Corollary~\ref{cor-main} same scale tangent cones along $\gamma$ vary H\"older continuously in $t$.
  Corollary~\ref{cor-main} implies that  H\"older continuity of tangent cones also holds in measure-metric sense with respect to the renormalized limit volume measures on the tangent cones.
This does not follow from the results of ~\cite{CoNa} which do not address measured continuity. Since   same scale tangent cones do not need to exists for all $t$ for any given scaling sequence, we state the H\"older continuity quanitatively using Sturm distance $\DD$ which metrizes the measured Gromov-Hausdorff topology on the class of spaces in question~\cite[Lemma 3.7]{Sturm06}.

\begin{cor}\label{mm-cont-1}
There exist $\eps=\eps(n,\delta, H)>0 , 0<\alpha(n)<1$ such that the following holds.

Let $M_j^n\to X$ where $\Ric_{M_j}\ge (n-1)H$.   Let $\gamma\co [0,1]\to X$ be a unit speed geodesic which is a limit of geodesics in $M_i$.
Then for any $t_1,t_2\in (\delta,1-\delta)$ with $|t_1-t_2|<\eps$
 we have that $$\DD((B_1(o_1),d\vol_1), (B_1(o_2),d\vol_2))\le |t_1-t_2|^{\alpha(n)}$$  where
 $(T_{\gamma(t_1)} X, d\vol_1), (T_{\gamma(t_2)} X, d\vol_2)$ are same scale tangent cones and 
$B_1(o_i)\subset T_{\gamma(t_i)} X$ is the unit ball around the vertex in $T_{\gamma(t_i)} X$.

\end{cor}
\begin{rmk}

Note that Bishop-Gromov volume comparison implies that in Corollary~\ref{cor-main} the set $C_i$ is  $C(n)|t_1-t_2|^{\alpha(n)/n}$ dense in $B_1(o_i)$ for $i=1,2$ and hence same scale tangent cones $T_{\gamma(t)}X$ are H\"older continuos in the pointed Gromov-Hausdorff topology. Of course, this is already known by ~\cite{CoNa}.
\end{rmk}

Let $X$ be a limit of $n$-manifolds with Ricci curvature bounded below. Recall that a point $p\in X$ is called $k$-regular if {\it every} tangent cone $T_pX$ is isometric to $\R^k$.
The collection of all $k$-regular points is denoted by $\Rg_k(X)$. (When the space $X$ in question is clear we will sometimes simply write  $\Rg_k$).

The set of regular points of $X$ is the union
\begin{align}
\Rg(X)\equiv \cup_k\Rg_k(X)\,.
\end{align}
The set of singular points $\mathcal S$ is the complement of the set of regular points.
It was proved in ~\cite{CC1} that $\vol(\mathcal S)=0$ with respect to any renormalized limit volume measure $d\vol$ on $X$.
Moreover, by ~\cite[Theorem 4.15]{CC3},   $\dim_{Haus} \Rg_k\le k$ and $d\vol$ is absolutely continuous on $\Rg_k(X)$  with respect to the $k$-dimensional Hausdorff measure. In particular,
\begin{equation}\label{eq-dim-reg}
\dim_{Haus} \Rg_k= k\quad \text{ if }
\vol(\Rg_k)>0.
\end{equation}
It was further shown in~\cite[Theorem 1.18]{CoNa} that there exists unique  integer $k$ such that
\begin{align}
\vol(\Rg_k)>0\, .
\end{align}
Altogether this implies that there exists unique integer $k$ such that
\begin{equation}\label{eq-vol-reg1}
\vol (X\backslash \Rg_k)=0\, .
\end{equation}
Moreover, it can be shown (Theorem~\ref{thm-dim} below) that this $k$ is equal to the largest integer $m$ for which $\Rg_m$ is non-empty.
Following Colding and Naber we will call this $k$ {\it the dimension of $X$} and denote it by $\dim X$.
(Note that it is not known to be equal to the Hausdorff dimension of $X$ in the collapsed case).

Corollary~\ref{cor-main} immediately implies
\begin{thm}\label{thm-main2}
Under the assumptions of Corollary~\ref{cor-main}  the dimension of same scale tangent cones $T_{\gamma(t)} X$ is constant for $t\in (0,1)$.
\end{thm}
\begin{proof}
For $t_2$ sufficiently close to $t_1=t$ let $\mathcal C_i\subset B_1(o_i)$ ($i=1,2$), $f\co \mathcal C_1\to \mathcal C_2$ be provided by Corollary ~\ref{cor-main}.
Let $k_i=\dim T_{\gamma(t_i)}X$. Suppose $k_1\ne k_2$, say $k_1<k_2$.  By using \eqref{eq-vol-reg1} and Corollary ~\ref{cor-main}~(ii) we can assume that $ C_i\subset \mathcal R_{k_i}(T_{\gamma(t_i)}X)$.
By above this means that $\dim_{Haus}\mathcal C_i=k_i$.

Since $f$ is Lipschitz  we have $\dim_{Haus}(f(\mathcal C_1))\le\dim_{Haus}\mathcal C_1=k_1$. Since $d\vol_2$ is absolutely continuous with respect to the $k_2$-dim Hausdorff measure on $\mathcal R_{k_2}$ and $k_2>k_1$  this implies that
$\vol_2 f(\mathcal C_1)=0$. This is a contradiction since $\vol_2 f(\mathcal C_1)=\vol_2 (\mathcal C_2)>0$.
\end{proof}

Note that a ``cusp" can exist in the limit space of manifolds with lower Ricci curvature bound, for example, a horn ~\cite[Example 8.77]{CC1}. Theorem \ref{thm-main2} indicates that a "cusp" cannot occur in the interior of limit geodesics. In particular, it provides a new way to rule out the trumpet ~\cite[Example 5.5]{CC2} and its generalizations ~\cite[Example 1.15]{CoNa}. Moreover, it shows that the following example cannot arise as a Gromov-Hausdorff limit of manifolds with lower Ricci bound, even through the tangent cones are H\"older (in fact, Lipschitz) continuous along the interior of  geodesics. This example cannot be ruled out by previously known results.

\begin{ex}
Let $Y=\Big\{(x,y,z)\in\R^3: z\ge \sqrt{x^4+|y|}-x^2\Big\}$.

\begin{figure}
\includegraphics[scale=0.35]{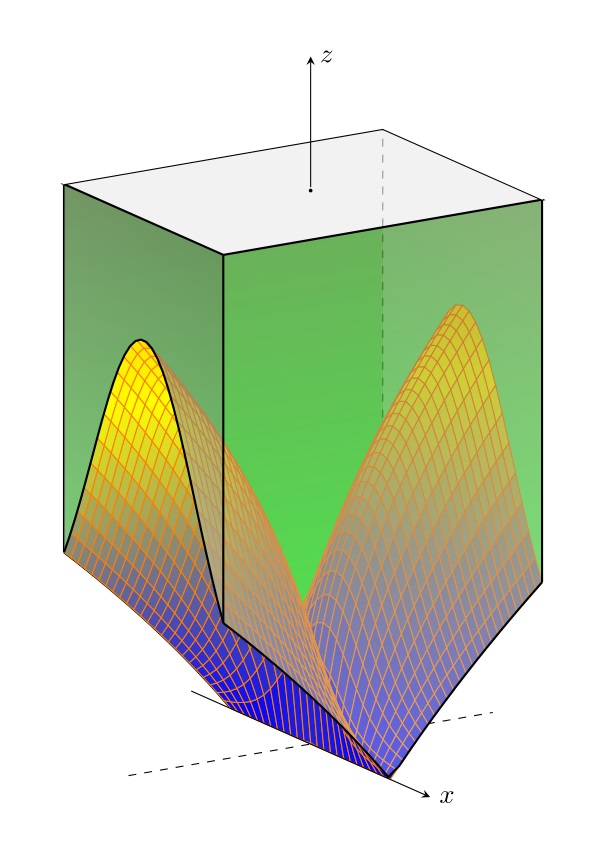}
\caption{$Y=\Big\{(x,y,z)\in\R^3: z\ge \sqrt{x^4+|y|}-x^2\Big\}$}
\end{figure}

Then $T_{(x,0,0)}Y=\Big\{(y,z)\in\mathbb R^2: z\ge \frac{|y|}{2x^2}\Big\}\times \R$ for $x\ne 0$ and $T_{(0,0,0)}Y=\R_+\times\R$.
Let $X$ be the double of $Y$ along its boundary. Then all points not on the $x$-axis are in $\Rg_3$ and along the $x$-axis
we have that for $x\ne 0$,  $T_{(x,0,0)}X=$  (double of $\Big\{(y,z)\in\mathbb R^2: z\ge \frac{|y|}{2x^2}\Big\}) \times \R$ (i.e. it's a cone $\times \R$) degenerating to
 $T_{(0,0,0)}X=\R_+\times\R$ .

 So $\dim T_{(0,0,0)}X=2$ but $\dim T_{(x,0,0)}X=3$ for $x\ne 0$. Lastly,
any segment of the geodesic $\gamma(t)=(t,0,0)$ is unique shortest between its end points and hence it's a limit geodesic if $Y$ is a limit of manifolds with $\Ric\ge -(n-1)H$. Hence Theorem  \ref{thm-main2} is applicable to $\gamma$ and therefore $X$ is not a limit of $n$-manifolds with $\Ric\ge -(n-1)H$.

Note that one can further smooth out the metric on $X$ along $\partial Y\backslash \{ x-$axis$\}$  to obtain a space  $X_1$ with similar properties  but which in addition is a smooth Riemannian manifold away from the $x$-axis. In particular $X_1$ is non-branching.

\end{ex}

\newpage

Next we want to mention several semicontinuity results about the Colding-Naber dimension  which further  suggest that this notion is a natural one.

Let  $\mathcal M^n$ be the space of pointed  Gromov-Hausdorff limits of manifolds with $\Ric\ge-(n-1)$.
Recall the following notions from~\cite{CC1}
\begin{defn}
Let $X\in \mathcal M^n$.
\begin{itemize}
 \item $\mathcal {WE}_k(X)=\{x\in X|$ such that {\rm some} tangent cone $T_xX$ splits off isometrically as $\R^k\times Y\}$. 
 \item $\mathcal E_k(X) = \{x \in X |$ such that   {\rm every} tangent cone $T_xX$ splits off isometrically as $\R^k\times Y\}$.
 \item $(\mathcal {WE}_k)_\eps(X)=\{x\in X|$ such that there exist $0<r\le 1$, $Y$ and $q\in \R^k\times Y$ such that $d_{G-H}(B_r(x),B_r^{\R^k\times Y}(q))<\eps r\}$.

 \end{itemize}
 \end{defn}

 By~\cite[Lemma 2.5]{CC1} there exists $\eps(n)>0$ such that if $p\in  (\mathcal {WE}_k)_\eps(X)$ for some $\eps\le \eps(n)$ then $\vol B_r(p)\cap\mathcal E_k>0$ for all sufficiently small $r$.

 Suppose  $(X_i, p_i)\in \mathcal M^n$, $(X_i, p_i)\to  (X,p)$ and $p\in \Rg_k(X)$. Then   $p\in (\mathcal {WE}_k)_{\eps(n)}(X)$ which obviously implies   that $p_i\in  (\mathcal {WE}_k)_{\eps(n)}(X_i)$ for all large $i$ as well.
By above this implies that  $\vol \mathcal E_k(X_i)>0$ for all large $i$.

 This together with ~\eqref{eq-vol-reg1} yields the following result of Honda proved in~\cite[Prop 3.78]{Honda} using very different tools.

 \begin{thm}~\cite[Prop 3.78]{Honda} \label{thm-dim-semicont}
Let $X_i\in\mathcal M^n$  and $\dim X_i=k$.
Let $(X_i,p_i)\overset{G-H}{\longrightarrow}  (X, p)$. Then $\dim X\le k$.
In other words, the dimension function is lower semicontinuous on $\mathcal M^n$ with respect to the Gromov-Hausdorff topology.
  \end{thm}
This theorem, applied to the convergence $(\frac 1{r}X,p)\underset{r\to 0}{\longrightarrow} (T_pX,o)=(\R^k,0)$ for $p\in\Rg_k$, immediately gives
 the following result which  also directly follows from ~\cite[Prop 3.1]{Honda1} and~\eqref{eq-vol-reg1}.
\begin{thm}\label{thm-dim}
Let $X\in \mathcal M^n$.  Then $\dim X$ is equal to the largest $k$ for which $\Rg_k(X)\ne \varnothing$.
\end{thm}



Another immediate consequence of Theorem \ref{thm-dim-semicont} is the following

\begin{cor}~\cite[Prop 3.78]{Honda} \label{cor-dim-tcone}
  Let $X\in \mathcal M^n$. Then for any $x\in X$  and any tangent cone $T_xX$ it holds that
  \[
  \dim T_xX\le \dim X\, .
  \]
\end{cor}

  It is obvious from ~\eqref{eq-dim-reg} that for any $X\in \mathcal M^n$ we have $\dim X\le \dim_{Haus} X$. However, as was mentioned earlier, the following natural question   remains open.
  \begin{quest}
  Let $X\in \mathcal M^n$. Is it true that $\dim X= \dim_{Haus} X$?
  \end{quest}

\subsection{Idea of the proof of Theorem~\ref{thm-main}}
Let  $\gamma\co[0,1]\to M^n$ be a unit speed shortest geodesic  in an $n$-manifold with $\Ric\ge (n-1)H$.
In ~\cite{CoNa} Colding and Naber constructed a parabolic approximation $h_\tau$ to $d(\cdot, p)$ given as the solution of the heat equation with initial conditions given by  $d(\cdot, p)$, appropriately cut off near the end points of $\gamma$ and outside a large ball containing $\gamma$. They showed that $h_{r^2}$ provides a good approximation to  $d^-=d(\cdot, p)$ on an $r$-neighborhood of $\gamma|_{[\delta,1-\delta]}$. In particular, they showed that
\begin{align}
  \int^{(1-\delta)}_{\delta }\left(\fint_{B_{r}(\gamma(t))} |\Hess_{h_{r^2}}|^2\right)\,dt \leq c(n,\delta,H)
  \label{CN.hess.int}
\end{align}
for all $r\le r_0(n,\delta,H)$. 
They used this to show that for any $t\in (\delta,1-\delta)$ most points in $B_r(\gamma(t))$ remain $r$-close to $\gamma$ under the reverse gradient flow of $d^-$ for a definite time $s\le\eps=\eps(n,\delta,H)$.
In section \ref{sec:grad-flow} we show that the same holds true for the reverse gradient flow $\phi_s$ of $h_{r^2}$.
Next,   the standard weak type 1-1 inequality for maximum function applied to the inequality (\ref{CN.hess.int}) implies that
\begin{align}
  \int^{(1-\delta)}_{\delta }\left(\fint_{B_{r}(\gamma(t))}(\Mx|\Hess_{h_{r^2}}|)^2\right)\,dt\leq c(n,\delta,H)
\end{align}
as well. This implies that for every $x$ in a subset $\mathcal C^{r}(\gamma(t))$ in $B_r(\gamma(t))$ of almost full measure the integral $\int_0^\eps \Mx |\Hess_{h_{r^2}}|(\phi_s(x))ds$ is small (see estimate (\ref{e.max.small})).
 Using a small modification of a lemma from ~\cite{Kap-Wil} this implies that for any such point $x$ and {\it any} $0<r_1\le r$ most points in $B_{r_1}(x)$ remain $r_1$-close to $\phi_s(x)$ for all  $s\le\eps$ under the flow $\phi_s$.
 This then easily implies that $\phi_s$ is Bilipschitz on $\mathcal C^{r}(\gamma(t))$ using Bishop-Gromov volume comparison and triangle inequality.

\subsection{Acknowledgements}

We are  very grateful to Aaron Naber for helpful conversations and to Shouhei Honda for bringing to our attention results of ~\cite{Honda1} and ~\cite{Honda}.
We are also very greateful to the referee for pointing out that our results imply Corollary~\ref{mm-cont-1}.
\section{Preliminaries}
In this section we will list most of the technical tools needed for the proof of Theorem~\ref{thm-main}.
 Throughout the rest of the paper, unless indicated otherwise,  we will assume that all manifolds $M^n$ involved are $n$-dimensional complete Riemannian satisfying
$$
\Ric_{M^n}\ge -(n-1)\,.
$$
\subsection{Segment inequality}

We will need the following result of Cheeger and Colding:

\smallskip

\begin{thm}[Segment inequality] ~\cite[Theorem 2.11]{CC}\label{seg-ineq} Given $n$ and $r_0>0$
there exists $c=c(n,r_0)$ such that the following holds.

Let $F\co M^n\to \R^+$ be a nonnegative measurable function.
Then for any $r\le r_0$ and $A,B\subset B_r(p)$ it holds
\[
\int_{A\times B}\int_{0}^{d(x,y)}F(\gamma_{x,y}(u))\,du\,d\vol_{x}\,d\vol_y\le c\cdot r\cdot (\vol A+\vol B)\int_{B_{2r}(p)}F(z)\,d\vol_z\,,
\]
where $\gamma_{z_1,z_2}$ denotes a minimal geodesic from $z_1$ to $z_2$.
\end{thm}
\subsection{Generalized Abresch-Gromoll Inequality}
 Let $\gamma\co [0,L]\to M$ be a minimizing  unit speed geodesic with $\gamma(0)=p,\gamma(L)=q$ where $L=d(p,q)$.
To simplify notations and exposition from now on we will assume  that $L=1$. Let $d^-=d(\cdot, p), d^+=d(\cdot, q)$, and let $e=d^++d^--d(p,q)$ be the excess function.

The following result is a direct consequence of \cite[Theorem 2.8]{CoNa} and, as was observed in~\cite{CoNa}, using the fact that $|\nabla e|\le 2$ it immediately implies the Abresch-Gromoll estimate~\cite{AbGr}.
\begin{thm}[Generalized Abresh-Gromoll Inequality] \cite[Theorem 2.8]{CoNa} \label{thm-gen-ab-gr}
There exist $c(n,\delta)$, $r_0(n,\delta)>0$ such that for any $0<\delta<t<1-\delta<1$, $0<r<r_0$ it holds
\[
\fint_{B_{r}(\gamma(t))} e\leq c(n,\delta)\,r^2\, .
\]

\end{thm}

\subsection{Parabolic approximation for distance functions}

Fix $\delta>0$ and let $h_t^\pm$ be  parabolic approximations to $d^\pm$ constructed in \cite{CoNa}.
They are given by the solutions to the heat equations
\[
\frac{d}{dt} h^\pm_t=\Delta h^\pm_t,\quad h^\pm_0(x)=\lambda(x) \cdot d^\pm(x)
\]
for appropriately constructed cutoff function $\lambda$.
 We will need the following properties of $h_t$ established in  \cite{CoNa}.
\begin{lem} \cite[Lemma 2.10]{CoNa}                    \label{lem:lap-comp}
There exists $c(n,\delta)$ such that
\begin{equation}
\Delta h^\pm_t \leq {c(n,\delta)}\, .
\end{equation}
\end{lem}

\begin{thm}   \cite[Theorem 2.19]{CoNa}                       \label{t:mainregthm1}
There exist $c(n,\delta), r_0(n,\delta)>0$ such that for all $r_1\leq r_0$  there exists $r\in [\frac{r_1}{2},2r_1]$ such that the following properties are satisfied
\begin{enumerate}
\renewcommand{\labelenumi}{(\roman{enumi})}
\item $|h^\pm_{r^2}-d^\pm|(x)\leq c\,r^2$ for any $x\in B_2(p)\backslash (B_\delta(p)\cup B_\delta(q))$ with $e(x)\leq r^2$
\label{e:emain1}
\item $\fint_{B_r(x)}||\nabla h^\pm_{r^2}|^2-1| \leq c\,r$.\label{e:emain2}
\item $\int^{(1-\delta)}_{\delta }\fint_{B_{r}(\gamma(t))}||\nabla h^\pm_{r^2}|^2-1| \leq c\,r^2$.
\label{e:emain3}
\item $\int^{(1-\delta)}_{\delta }\fint_{B_{r}(\gamma(t))}|\Hess_{h^\pm_{r^2}}|^2\leq c$.
\label{e:emain4}
\end{enumerate}
\end{thm}

\subsection{First Variation formula}
We will need the following  lemma  (cf. \cite[Lemma 3.4]{CoNa} ).
\begin{lem}\label{lem:dist-growth}

Let $X$ be a smooth vector field on $M$  and  let $\sigma_1(t),\sigma_2(t)$ be smooth curves. Let $p=\sigma_1(0), q=\sigma_2(0)$. Then
\[
\left|\frac{d^+}{dt}d(\sigma_1(t),\sigma_2(t))|_{t=0}\right| \leq |X(p)-\sigma_1'(0)|+|X(q)-\sigma_2'(0)|+ \int_{\gamma_{p,q}}|\nabla_\cdot  X|\,,
\]
where $\gamma_{p,q}\co [0,d(p,q)]\to M$ is a shortest geodesic from $p$ to $q$.
Here $|\nabla_\cdot  X|$ means the norm of the  full covariant derivative of $X$ i.e. norm of the map $v\mapsto \nabla_vX$.
In particular, if  $h\co M\to\R$ is smooth and $X=\nabla h$, then
\[
\left|\frac{d^+}{dt}d(\sigma_1(t),\sigma_2(t))|_{t=0}\right| \leq |X(p)-\sigma_1'(0)|+|X(q)-\sigma_2'(0)|+ \int_{\gamma_{p,q}}|\Hess_h|\,.
\]

\end{lem}
\begin{proof}
The lemma easily follows from the first variation formula for distance functions and the triangle inequality.
\end{proof}

\subsection{Maximum function}
Let  $f\co M\to\R$  be a nonnegative function. Consider the maximum function $\Mx_{\rho} f(p):=\sup_{r\le
\rho}\sfint_{B_r(p)}f$ for $\rho\in (0,4]$. We'll set $\Mx f:=\Mx_1 f$.

The following lemma is well-known \cite[p. 12]{Stein}.
 \begin{lem}[Weak type 1-1 inequality]\label{lem:mf-est}\label{lem: weak11}
Suppose $(M^n,g)$ has $\Ric\ge -(n-1)$ and let $f\co M\to\R$  be a nonnegative function. Then the following holds.
\begin{enumerate}
\renewcommand{\labelenumi}{(\roman{enumi})}
\item If $f\in L^\a(M)$ with $\a\ge 1$ then $\Mx_\rho f$ is finite almost everywhere.
\item If $f\in L^1(M)$ then $\vol\bigl\{x\in M\co \Mx_\rho f(x)>c\bigr\}\le \frac{C(n)}{c} \int_M f$ for any $c>0$.
\item If $f\in L^\a(M)$ with $\a>1$ then $\Mx_\rho f\in L^\a(M)$ and $||\Mx_\rho f||_\a\le C(n,\a)||f||_\a$.
\end{enumerate}
\end{lem}
This lemma easily generalizes to functions defined on subsets as follows:
\begin{cor}\label{cor:Mx}
Let $\Ric_{M^n}\ge -(n-1)$ and $f\co M\to \R_+$ be measurable. Let $A\subset M$ be measurable such that $f\in L^\alpha(U_\rho(A))$ where $\alpha>1$. Here $U_\rho(A)$ denotes the $\rho$-neighborhood of $A$ .
Then
\[
||\Mxrh f||_{L^\alpha(A)}\le C(n,\alpha)||f||_{L^\alpha(U_\rho(A)) }\,.
\]
\end{cor}
\begin{proof}
Let $\bar f=f\cdot \chi_{U_\rho(A)}$. Obviously, $\Mxrh f(x)=\Mxrh \bar f(x)$ for any $x\in A$. The result follows by applying Lemma ~\ref{lem:mf-est} (iii) to $\bar f$.
\end{proof}

\section{Gradient flow of the parabolic approximation}\label{sec:grad-flow}

Let $\phi_s$ be the {\it reverse}  gradient flow of $h=h_{r^2}^-$ (i.e. the gradient flow of $-h_{r^2}^-$) and let $\psi_s$ be the reverse gradient flow of $d^-$. We first want to show that for most points  $x\in B_r(\gamma(t))$ we have that $\phi_s(x)\in B_{2r}(\gamma(t-s))$ 
 for all $t\in(\delta,1-\delta)$ and $s\in [0,\eps]$ for some uniform $\eps=\eps(n,\delta)$.

Note that this (and more) is already known for $\psi_s$ by ~\cite{CoNa}.
Following Colding-Naber we use the following
\begin{defn}\label{defn:A-B}
 For $0<s<t<1$ define the set $\mathcal{A}^t_s(r)\equiv \{z\in B_r(\gamma(t)):\psi_u(z)\in B_{2r}(\gamma(t-u)),\,\text{  } \forall 0\leq u\leq s\}$.
 Similarly, we define $\mathcal{B}^t_s(r)\equiv \{z\in B_r(\gamma(t)):\phi_u(z)\in B_{2r}(\gamma(t-u)),\,\text{  }\, \forall 0\leq u\leq s\}$.
\end{defn}

An important technical tool used to prove the main results of ~\cite{CoNa} is the following
\begin{prop} \cite[Proposition 3.6]{CoNa}\label{prop-dist-grad}
There exist $r_0(n,\delta)$ and $\epsilon_0(n,\delta)$ such that if $t\in(\delta,1-\delta)$ and $\eps\leq \epsilon_0$ then $\forall r\leq r_0$ as in Theorem \ref{t:mainregthm1}
 we have
\[
\frac{1}{2}\leq \frac{\vol(\mathcal{A}^{t}_{\eps}(r))}{\vol(B_r(\gamma(t)))}\,.
\]

\end{prop}

Unlike Colding-Naber we prefer to work with the gradient flow of the parabolic approximation $h$ rather than the gradient flows of $d^\pm$, because the gradient flow of $h$ provides better distance distortion estimates since in that case the two terms outside the integral in Lemma~\ref{lem:dist-growth}  vanish and the resulting inequality scales better in the estimates involving maximum function (see Lemma ~\ref{lem:dtr} below). Therefore, our first order of business is to establish the following lemma which says that Proposition \ref{prop-dist-grad} holds for the gradient flow of $-h$ as well:

\begin{lem}\label{lem: distort1}
There exists $r_1(n,\delta)$ and $\epsilon_1(n,\delta)$ such that if $\delta<t-\eps<t<1-\delta$ and  $\eps\leq \eps_1$ then $\forall r\leq r_1$
 we have
\[
\frac{1}{2}\leq \frac{\vol(\mathcal{A}^{t}_{\eps}(r))}{\vol(B_r(\gamma(t)))}
\]
and
\[
\frac{1}{2}\leq \frac{\vol(\mathcal{B}^{t}_{\eps}(r))}{\vol(B_r(\gamma(t)))}\,.
\]

\end{lem}
The proof of Proposition \ref{prop-dist-grad}  uses bootstrapping in $\eps, r$ starting with infinitesimally small (depending on $M$!) $r$ (cf. Lemma~\ref{lem:dtr} below) for which the claim easily follows from Bochner's formula applied to $d^-$ along $\gamma$.
We don't utilize bootstarpping in $r$ and instead use that the result has already been established for the gradient flow of $-d^-$.

\begin{proof}
Of course, we only need to prove the second inequality as the first one holds by  Proposition \ref{prop-dist-grad} for some  $r_0(n,\delta), \epsilon_0(n,\delta)>0$. By possibly making $r_0$ smaller we can ensure that it satisfies Theorem \ref{t:mainregthm1}.

Let $0<\eps<\eps_0$ be small (how small it will be chosen later). Let
\begin{equation}
S_t\equiv\left\{0\le s<t-\delta: \frac{1}{2}< \frac{\vol(\mathcal{B}^{t}_{s}(r))}{\vol(B_r(\gamma(t)))} \right\}\,.
\end{equation}
We wish to show that $S_t$ contains $[0,\eps]$ for some uniform $\eps=\eps(n)$.
Obviously $S_t$ is open in $[0,\eps]$  so it's enough to show that it's also closed.
To establish  this it's enough to show that if $\eps'\le\eps$ and  $[0,\eps')\subset S_t$ then $\eps'\in S_t$.

For any $0<s<t$ we define $\tilde c_s^t$ to be the characteristic function of the set $\mathcal{A}_s^t(r)\times\mathcal{B}_s^t(r)$.
The same argument as in ~\cite{CoNa} shows that
\begin{align}\label{eq-col-nab1}
 &\fint_{B_r(\gamma(t))\times B_r(\gamma(t))}\tilde c_s^t(x,y)\,\left(\int_{\gamma_{\psi_s(x),\phi_s(y)}}|\Hess_h|\right)\,d\vol_x\, d\vol_y
\\
 &\qquad\leq C(n,\delta)\,r\,\left(\frac{\vol(B_r(\gamma(t-s)))}{\vol(B_r(\gamma(t)))}\right)^{2} \fint_{B_{5r}(\gamma(t-s))}|\Hess_h|\,.\notag
\end{align}

Indeed, we have
\begin{align}
  &\int_{B_r(\gamma(t))\times B_r(\gamma(t))}\tilde c_s^t(x,y)\left(\int_{\gamma_{\psi_s(x),\psi_s(y)}}|\Hess_h|\right)\,d\vol_x\, d\vol_y \\
  &= \int_{\mathcal{A}_s^t(r)\times\mathcal{B}_s^t(r)} \left(\int_{\gamma_{\psi_s(x),\psi_s(y)}}|\Hess_h|\right) \,d\vol_x\, d\vol_y
  \notag\\
  &\leq C(n,\delta)\int_{\psi_s(\mathcal{A}_s^t(r))\times\phi_s(\mathcal{B}_s^t(r))} \left(\int_{\gamma_{x,y}}|\Hess_h|\right)\,d\vol_{\bar x}\, d\vol_{\bar y}\,,
  \notag
\end{align}
where the last inequality follows from the fact that  $\Delta h\le c(n,\delta)$ by Lemma~\ref{lem:lap-comp} and hence the Jacobian of $\phi_s$ satisfies
\begin{equation}\label{eq:jac-bnd1}
J_{\phi_s}\ge e^{C(n,\delta)s}\,.
\end{equation}
Similar inequality holds for $\psi_s$ by Bishop-Gromov volume comparison.
 Since $\psi_s(\mathcal{A}_s^t(r))$, $\phi_s(\mathcal{B}_s^t(r)) \subseteq B_{2r}(\gamma(t-s))$ by definition, by  the segment inequality (Theorem~\ref{seg-ineq} )  we have
 \begin{align}
    &\int_{\psi_s(\mathcal{A}_s^t(r))\times\phi_s(\mathcal{B}_s^t(r))} \left(\int_{\gamma_{x,y}}|\Hess_h|\right)\,d\vol_{\bar x}\, d\vol_{\bar y}
    \\
    &\leq C(n,\delta)\,r\,[\Vol(\psi_s(\mathcal{A}_s^t(r)))+\Vol(\phi_s(\mathcal{B}_s^t(r)))] \int_{B_{5r}(\gamma(t-s))}|\Hess_h|
    \notag\\
    &\leq C(n,\delta)\,r\,\Vol(B_{5r}(\gamma(t-s)))\int_{B_{5r}(\gamma(t-s))}|\Hess_h|
    \notag\\
    &=C(n,\delta)\,r\,\Vol(B_{5r}(\gamma(t-s)))^2\fint_{B_{5r}(\gamma(t-s))}|\Hess_h|
    \notag\\
    &\le C(n,\delta)\,r\,\Vol(B_{r}(\gamma(t-s)))^2\fint_{B_{5r}(\gamma(t-s))}|\Hess_h|\,,
    \notag
\end{align}
where the last inequality follows by Bishop-Gromov.
Thus,
\begin{align}
  &\int_{B_r(\gamma(t))\times B_r(\gamma(t))}\tilde c_s^t(x,y)\left(\int_{\gamma_{\psi_s(x),\psi_s(y)}}|\Hess_h|\right) \,d\vol_x\, d\vol_y
  \\
  &\qquad\le C(n,\delta)\,r\,\Vol(B_{r}(\gamma(t-s)))^2\fint_{B_{5r}(\gamma(t-s))}|\Hess_h|\,.\notag
\end{align}
Dividing by $\Vol(B_{r}(\gamma(t)))^2$  we get \eqref{eq-col-nab1}. By \cite[Cor 3.7]{CoNa} we have that
\begin{equation}\label{eq-vol-comp1}
C^{-1}\leq \frac{\vol(B_r(\gamma(t-s)))}{\vol(B_r(\gamma(t)))} \leq C\, .
\end{equation}
for some universal $C=C(n,\delta)$ and therefore
\begin{align}
  &\fint_{B_r(\gamma(t))\times B_r(\gamma(t))}\tilde c_s^t(x,y)\,\left(\int_{\gamma_{\psi_s(x),\phi_s(y)}}|\Hess_h|\right)\,d\vol_x\, d\vol_y
  \label{eq1}\\
  &\qquad\leq C(n,\delta)\,r \fint_{B_{5r}(\gamma(t-s))}|\Hess_h|\,.\notag
\end{align}

Let
\begin{align}
 \tilde I^r_\eps\equiv \fint_{B_r(\gamma(t))\times B_r(\gamma(t))}
 \int_0^\eps \tilde c_s^t(x,y)
 \left(\int_{\gamma_{\psi_s(x),\phi_s(y)}}|\Hess_{h}|\right)
 \,ds\,d\vol_x\,d\vol_y\,.
\end{align}
Then by \eqref{eq1} and  Theorem~\ref{t:mainregthm1} we have that
\begin{align}\label{eq:I}
  \tilde I^r_{\eps'}
  &=\int_0^{\eps'}\fint_{B_r(\gamma(t))\times B_r(\gamma(t))} \tilde c_u^t(x,y) \left(\int_{\gamma_{\psi_u(x),\phi_u(y)}} |\Hess_{h}|\right)\,d\vol_x\,d\vol_y\,ds
  \\
  &\leq C(n,\delta)\,r \int_0^{\eps'}\left(\fint_{B_{5r}(\gamma(t-s))}|\Hess_{h}|\,d\vol \right)ds
  \notag\\
  &\leq C(n,\delta)r \sqrt{\eps'}\left(\int_\delta^{1-\delta} \fint_{B_{5r}(\gamma(s))}|\Hess_{h}|^2\,d\vol\,ds\right)^{1/2}
  \notag\\
  &\leq C(n,\delta)\,\sqrt{\eps'}\,r\, . \notag
\end{align}

Let
\[
\tilde T^r_{\eta}\equiv \left\{x\in B_r(\gamma(t)): x\in \mathcal{A}^{t}_{\eps}(r) \text{ and } \fint_{\{x\}\times B_r(\gamma(t))}\int_0^{\eps'} \tilde c_s^t(x,y)\left(\int_{\gamma_{\psi_s(x),\phi_s(y)}}|\Hess_{h}|\right)\leq \eta^{-1}\tilde I^r_{\eps'}\right\}\, ,
\]
and for $x\in \tilde T^r_\eta$ let us define
\begin{equation}\label{eq:Tx}
\tilde T^r_\eta(x) \equiv \left\{y\in B_r(\gamma(t)): \int_0^{\eps'} \tilde c_s^t(x,y)\left(\int_{\gamma_{\psi_s(x),\phi_s(y)}}|\Hess_{h}|\right)ds\leq \eta^{-2}\tilde I^r_{\eps'}\right\}\, .
\end{equation}

Here $\eta=\eta(n,\delta, d(p,q))>0$ is small and chosen first. Then $\eps$ is chosen later depending on $\eta$.
By \cite[page 34, equations (115) (117)]{CoNa} we  can assume that
\begin{equation}
\frac{\Vol(\mathcal{A}^{t}_{\eps}(r))}{\Vol(B_r(\gamma(t)))}\geq 1-C(n,\delta)\eta\,.
\end{equation}
if $\eps\le C(n,\delta)\eta^{\alpha(n)}$ for some universal $\alpha(n)>1$.
Therefore, by construction we have that
\begin{equation}\label{eq:volratio1}
\frac{\Vol(\tilde T^r_\eta)}{\Vol(B_r(\gamma(t)))}\geq 1-C(n,\delta)\eta\, ,
\end{equation}
and hence
\begin{equation}\label{eq:volratio2}
\frac{\Vol(\tilde T^r_\eta(x))}{\Vol(B_r(\gamma(t)))}\geq 1-C(n,\delta)\eta,\qquad\forall x\in T^r_\eta\, .
\end{equation}
We choose $\eta$ so that $ C(n,\delta)\eta\ll 1$ in \eqref{eq:volratio1} and \eqref{eq:volratio2}.

Let $x\in \tilde T^r_\eta\cap \tilde T^{r/100}_\eta$ (this intersection is non-empty for small $\eta=\eta(n)$ by Bishop-Gromov) and let $y\in \tilde T^r_\eta(x)$.
We will fix $\eta>0$ satisfying the above conditions from now on. We claim that then $y\in \mathcal{B}^{t}_{\eps'}(r)$.

Indeed $d(\psi_s(x),\gamma(t-s))\le r/50$ for all $s\le \eps'$ since $x\in  \tilde T^{r/100}_\eta\subset \mathcal{A}^{t}_{\eps}(r/100)$ .  So by the triangle inequality it's enough to show that $d(\psi_s(x),\phi_s(y))\le 1.1r$ for any $s\le\eps'$.

Let $S=\{s\le \eps':  y \in \mathcal{B}^{t}_{s}(r)\}$. This set is obviously open and connected in $[0,\eps']$. We claim that $S=[0,\eps']$.
Let $\bar s=\sup\{s\co s\in S\}$.

Note that for any $0<s<\bar s$ we have that $\tilde c_s^t(x,y)=1$.
Therefore, by \eqref{eq:I} and \eqref{eq:Tx} for any $0<s<\bar s$ we have
\begin{equation}\label{eq:hess-int1}
\int_0^{ s} \left(\int_{\gamma_{\psi_u(x),\phi_u(y)}}|\Hess_{h}|\right)du\leq \eta^{-2}\tilde I^r_{t-\bar s}\le \frac{C(n,\delta)}{\eta^2}\sqrt{\eps}r\le 0.001 r
\end{equation}
if $\eps=\eps(\eta)$ is chosen small enough.

Next, recall that by \cite[ Lemma 2.20(3)]{CoNa} we have that for any $x\in B_r(\gamma(t))$,
\begin{equation}
\int_{0}^{s}|\nabla h(\psi_u(x))-\nabla d^{-}(\psi_u(x))|du\leq {c(n,\delta)\sqrt{s}}\,
 (\sqrt{e(x)}+r)\,.
\end{equation}
Further, by Theorem~\ref{thm-gen-ab-gr}
we know that
\begin{equation}
\fint_{B_{r}(\gamma(t))} e\leq c(n,\delta)\,r^2\, .
\end{equation}
Therefore, without losing generality by making the sets $\tilde T^r_\eta$ slightly smaller we can assume that for any $x\in \tilde T^r_\eta$ we have
\begin{equation}
e(x)\le \eta^{-1}r^2\,.
\end{equation}
Thus, for all $x\in \tilde T^r_\eta$ we have
\begin{equation}\label{eq:grad-int1}
\int_{0}^{s}|\nabla h (\psi_u(x)-\nabla d^{-}(\psi_u(x)|du\leq {c(n,\delta)\sqrt{s}}\cdot(\eta^{-1/2} r+r)<0.001r
\end{equation}
if $\eps=\eps(n,\delta,\eta)$ is small enough. Therefore, by Lemma~\ref{lem:dist-growth} and using \eqref{eq:hess-int1} and \eqref{eq:grad-int1} we get that
\begin{equation}
  d(\psi_s(x),\phi_s(y))\le 0.002r+d(x,y)<1.1r\,.
\end{equation}
By the triangle inequality,
\begin{equation}
  d(\phi_s(x),\gamma(t-s))\le r/50+1.1r\le1.5r< 2r\,.
\end{equation}

By continuity the same holds for $\bar s$ and hence $\bar s\in S$. Thus $S$ is both open and closed in $[0,\eps']$ and therefore $S=[0,\eps']$.
Unwinding this further we see that this means that $\tilde T^r_\eta(x)\subset  \mathcal{B}^{t}_{\eps'}(r)$.
Therefore, by \eqref{eq:volratio2}
\begin{equation}
\frac{1}{2}\leq \frac{\vol(\mathcal{B}^{t}_{\eps'}(r))}{\vol(B_r(\gamma(t)))}
\end{equation}
when $\eta$ was chosen small enough so that $C(n,\delta)\eta<1/2$.
Hence $\eps'\in S_t$. Therefore, $S_t$ is both open and closed and $\eps'=\eps$.
\end{proof}
The proof of Lemma~\ref{lem: distort1} shows that $\tilde T^r_\eta(x)\subset  \mathcal{B}^{t}_{\eps}(r)$ for appropriately chosen $\eps$ depending on $\eta$.
Moreover, the proof shows that $\eps$ can be chosen to be of the form $\eps=C(n,\delta)\eta^{\alpha(n)}$ for some  $\alpha(n)>1$.
 In view of \eqref{eq:volratio1} this means that the conclusion can be strengthened as follows (cf. \cite{CoNa})
\begin{lem}\label{lem: distort2}
For every $\eta\leq \eta_0(n,\delta)$ and $r\leq r_0(n,\delta)$ that there exists $\eps\equiv \eps(n,\eta,\delta)$ such that the set
$\mathcal{B}^t_\eps(r)\equiv \{z\in B_{r}(\gamma(t)):\phi_s(z)\in B_{2r}(\gamma(t-s))\,\text{  }\, \forall 0\leq s\leq \eps\}$
satisfies
\begin{align*}
 \frac{\vol \mathcal{B}^t_\eps(r)}{\Vol(B_r(\gamma(t)))} \geq 1-C(n,\delta)\,\eta\, .
\end{align*}
Moreover, $\eps(n,\eta,\delta)$ can be chosen to be of the form $\eps=\tilde C(n,\delta)\eta^{\alpha(n)}$ for some  $\alpha(n)>1$.
\end{lem}
When $\eta$ is sufficiently small this means that most points in $B_r(\gamma(t))$ remain close to the geodesic $\gamma$ under the flow $\phi_s$. Also, provided $C(n,\delta)\eta<1/2$  using Bishop-Gromov,  the above lemma, \eqref{eq:jac-bnd1} and \eqref{eq-vol-comp1} give the following

\begin{lem}\label{lem-comp-av1}
Let $F\co M\to\R$ be a nonnegative measurable function. Then

\[
\fint_{\mathcal{B}^t_\eps(r)}F(\phi_s(x))\,d\vol_x
\le C(n,\delta)\fint_{ B_{2r}(\gamma(t-s))}F(\bar x) \,d\vol_{\bar x}
\]
for any $s\le\eps$ as in Lemma \ref{lem: distort2}.
\end{lem}
\begin{rmk}
It is obvious that all results of this section concerning the flow of $-h^-_{r^2}$ are also true for the flow of $-h^+_{r^2}$.
\end{rmk}
\section{Bilipschitz control}\label{sec:bilip}
The goal of this section is to prove the following equivalent version of Theorem \ref{thm-main}

\begin{thm}\label{thm-main1}
Given $H\in\R, 0<\delta<1/3,0<\eta<1$
there exist $ r_0(n,\delta,H)$, $\eps=C(n,\delta,H)\eta^{\alpha(n)}$ where $\alpha(n)>1$ such that the following holds:

Suppose $(M^n,g)$ is complete with $\Ric_M\ge (n-1)H$ and let $\gamma\co [0,1]\to M^n$ be a unit speed minimizing geodesic.
Then for any $t_1,t_2\in (\delta,1-\delta)$ with $|t_1-t_2|<\eps$ and any $r<r_0$ there are subsets $\mathcal C_i^r\subset B_r(\gamma(t_i))$   ($i=1,2$)  such that
\[
\frac{\vol \mathcal C_i^r}{\vol B_r(\gamma(t_i))}\ge 1-\eta
\]
and there exists a $(1+C(n,\delta,H)\eta^{\beta(n)})$-Bilipschitz map $f_r\co \mathcal C_1^r\to \mathcal  C_2^r$ for some  $0<\beta(n)<1$.
\end{thm}

As before, to simplify notation we will assume that $H=1$ and $\Ric_{M^n}\ge -(n-1)$.

Corollary~\ref{cor:Mx} essentially means that all estimates involving integrals of $|\Hess_h|$ from the previous section remain true for $\Mx_\rho |\Hess_h|$.
In particular, for $r_0$ as in Theorem  ~\ref{t:mainregthm1}
and any $r\le r_0/10$ we have
\begin{equation}\label{eq:hess2}
\int_{\delta}^{1-\delta}\left(\fint_{B_{4r}(\gamma(t))}|\Hess_{h_{16r^2}}|^2\right)\,dt\leq \frac{C}{\delta}\, .
\end{equation}
By Corollary~\ref{cor:Mx} this implies that
\begin{equation}\label{eq:mx-hess1}
\int_{\delta}^{1-\delta}\left(\fint_{B_{2r}(\gamma(t))}(\Mxr |\Hess_h|)^2\right)\,dt\leq \frac{C}{\delta}\, .
\end{equation}
where $h=h^-_{16r^2}$. It is clear that all  results from the previous section work for this $h$ as well as $h^-_{r^2}$.
Therefore, everywhere in the previous section where we used \eqref{eq:hess2} we could have used \eqref{eq:mx-hess1} instead. Indeed, we have for any $2\delta<t<1-2\delta$,
\begin{align}
  &\int_0^\eps\left(\fint_{B_{2r}(\gamma(t-s))}\Mxr |\Hess_h|\right)\,ds
\\
  &\le \sqrt\eps \cdot\int_0^\eps\left(\fint_{B_{2r}(\gamma(t-s))}\Mxr |\Hess_h|\right)^2\,ds
  \notag\\
  &\le \sqrt\eps \cdot \int_0^\eps\fint_{B_{2r}(\gamma(t-s))}( \Mxr |\Hess_h|)^2\, ds
  \le C(n,\delta) \sqrt\eps\,.\notag
\end{align}
In view of Lemma~\ref{lem-comp-av1} this  implies that

\begin{equation}
\fint_{\mathcal{B}^t_\eps(r))}\int_0^\eps  \Mxr |\Hess_h(\phi_s(x))|\,ds\le C(n,\delta) \sqrt\eps\,.
\label{e.max.small}
\end{equation}
This means that for most points $x\in \mathcal{B}^t_\eps(r)$ the integral $\int_0^\eps \Mxr |\Hess_h(\phi_s(x))|\,ds$ is bounded. More precisely, given any  $0<\nu<1$
let
\begin{equation}\label{eq:hess-bnd-most-pts1}
 \mathcal{B}^t_\eps(r,\nu)\equiv\left\{x\in  \mathcal{B}^t_\eps(r):\quad \int_0^\eps \Mxr |\Hess_h(\phi_s(x))|\,ds\le \frac {C(n,\delta)\sqrt\eps} {\nu}\right\}\,.
\end{equation}
 Then
\begin{equation}\label{eq:hess-bnd-most-pts2}
 \frac{\vol \mathcal{B}^t_\eps(r,\nu) }{\vol \mathcal{B}^t_\eps(r)}>1-\nu\,.
\end{equation}

We will need  the following slight modification of Lemma 3.7 from~\cite{Kap-Wil}
\begin{lem}\label{lem:dtr}
Given $c>0$, there exists (explicit) $C=C(n,\lambda)$ such that the following holds.
Suppose  $(M^n,g)$ has $\Ric_{M^n}\ge -(n-1)$ and $X^t$ is a vector field
with compact support, which depends on time  but is piecewise constant in time.
Let $c(t)$ be the integral curve of $X^t$ with $c(0)=p_0\in M^n$
and assume that $\divr X^t\ge -\lambda$ 
on $B_{10}(c(t))$ for all $t\in [0,1]$.

Let $\phi_t$ be the flow of $X^t$. Define the distortion function  $\dt_r(t)(p,q)$
of the flow
on scale $r$ by the formula
\begin{align}
\dt_r(t)(p,q):=\min\Bigl\{r,\;\max_{0\le\tau\le t}\bigm|d(p,q)-d(\phi_\tau(p),\phi_\tau(q))|\Bigr\}\,.
\end{align}
Put
$\mu=\int_0^1\Mx_1(\|\nabla_\cdot X^t\|)(c(t))\,dt$. Then for any $r\le 1/10$ we have
\begin{eqnarray}\label{eq:aver1}
\fint_{B_r(p_0)\times B_r(p_0)}\dt_r(1)(p,q)\,d\vol_p\, d\vol_q\le C r\cdot \mu
\end{eqnarray}
and there exists $B_r(p_0)'\subset B_r(p_0)$ such that
\begin{align}
  \frac{\vol(B_r(p_0)')}{\vol(B_r(p_0))}\ge1-C\mu
\end{align}
and $\phi_t(B_r(p_0)')\subset B_{2r}(c(t))$.


\end{lem}
This lemma immediately implies
\begin{cor}\label{cor-good-pts}
Under the assumptions of Lemma~\ref{lem:dtr} for every $r\le 1/10$ there exists
$B_r(p_0)''\subset  B_r(p_0)$ such that
\[
\frac{
\vol(B_r(p_0)'')}{\vol(B_r(p_0))}\ge 1-C(n,\lambda)\mu
\]
\[
\phi_t(B_r(p_0)'')\subset B_{2r}(c(t)) \text{  for all } t\in [0,1]
\]
and
\[
\forall p,q\in B_r(p_0)'',\quad \dt_r(1)(p,q)\le C(n,\lambda) r\cdot \mu\,.
\]

\end{cor}
In~\cite{Kap-Wil} Lemma~\ref{lem:dtr} is stated for divergence free vector fields.
We want to apply it to $X=-\nabla h$ which is not divergence free.  However, recall that by Lemma~\ref{lem:lap-comp}  it does satisfy $\divr X\ge -\lambda(n,\delta)$ and hence the Jacobian of its flow map satisfies
\begin{equation}\label{eq:jacb2}
J_{\phi_s}\ge e^{-\lambda(n,\delta) s}
\end{equation}
pointwise.

The proof given in ~\cite{Kap-Wil} goes through verbatim with a straightforward change in one place using \eqref{eq:jacb2} instead of the flow of $X$ being volume-preserving.
We include the proof here for reader's convenience.

\begin{proof}[Proof of Lemma~\ref{lem:dtr}] We prove the statement for a constant in time
vector field $X^t$. The general case is completely analogous
except for additional notational problems.

Notice that all estimates are trivial if $\mu\ge \tfrac{2}{ C}$.
Therefore it suffices to prove the statement with a universal constant $C(n,\lambda)$
for all $\mu\le \mu_0(n,\lambda)$.   We put $\mu_0=\frac1{2C}$ and determine $C\ge 2$ in the process.
We  proceed by induction on the size of $r$.

Notice that the differential of
$\phi_s$ at $c(0)$  is Bilipschitz
with Bilipschitz constant
\begin{align}
  e^{\int_0^s \|\nabla_{\cdot} X\|(c(t))dt}\le 1+2\mu\,.
\end{align}
Thus the Lemma holds
for very small $r$.

Suppose the result holds for some $r/10\le 1/100$.
It suffices to prove that it then holds for $r$.
By induction assumption we know that for any $t$ there exists
$B_{r/10}(c(t))'\subset B_{r/10}(c(t))$ such that for any $s\in [-t,1-t]$ we have
\begin{align}
\vol (B_{r/10}(c(t))')\ge (1-C\mu)\vol (B_{r/10}(c(t)))\ge \frac{1}{2}\vol (B_{r/10}(c(t)))
\end{align}
and
\begin{align}
\phi_s(B_{r/10}(c(t))')\subset B_{r/5}(c(t+s))\,,
\end{align}
where  we used $\mu \le \tfrac{1}{2C}$ in the inequality.
This easily implies that $\vol (B_{r/10}(c(t)))$ are comparable for all $t$.
More precisely, for  any $t_1, t_2\in [0,1]$ we have that
\begin{equation}\label{eq:comp1}
\tfrac{1}{C_0}\vol B_{r/10}(c(t_1))\le \vol B_{r/10}(c(t_2))\le C_0\vol B_{r/10}(c(t_1))
\end{equation}
with a computable universal $C_0=C_0(n)$.
Put
\begin{align}
h(s)=\fint _{B_{r/10}(c(0))'\times B_r(c(0))}\dt_r(s)(p,q)\,\,d\vol_p\,d\vol_q\,,
\end{align}
\begin{eqnarray}
  U_s&:=&\{(p,q)\in B_{r/10}(c(0))'\times B_r(c(0))\mid \dt_r(s)(p,q)<r\}\,,
  \\
  \phi_s(U_s)&:=&\{(\phi_s(p),\phi_s(q))\mid (p,q)\in U_s\}\,,\hspace*{1em}\mbox{and}
  \notag\\
  \dt_r'(s)(p,q)&:=&\limsup_{h\searrow 0} \tfrac{\dt_r(s+h)(p,q)-\dt_r(s)(p,q)}{h}\,.
  \notag
\end{eqnarray}
As $\dt_r(t)\le r$ is monotonously increasing,
we deduce that if $\dt_r(s)(p,q)= r$, then $\dt_r'(s)(p,q)=0$.
Since $\dt_r(s+h)(p,q)\le
\dt_r(s)(p,q)+\dt_r(h)(\phi_s(p),\phi_s(q))$ and $\phi_s$
satisfies $J_{\phi_t}\ge e^{-\lambda t}$
it follows
\begin{eqnarray}
  h'(s)\le\fint_{U_s}\dt_r'(\phi_1(x),\phi_s(y))&\le&e^{s\lambda} \fint_{\phi_s(U_s)}\dt_r'(0)(p,q)\\
  &\le&e^{s\lambda} \tfrac{4\vol B_{3r}(c(s))^2}{\vol B_{r/10}(c(0))^2}\fint_{B_{3r}(c(s))^2}
\dt_r'(0)(p,q)\,,\notag
\end{eqnarray}
where we used that $\phi_s(B_{r/10}(p_0)')^2\subset \phi_s(U_s)\subset B_{3r}(c(s))^2$. We would like to point out that using $J_{\phi_s}\ge e^{-\lambda s}$ instead of $J_{\phi_s}=1$ (which is true for flows of harmonic maps)
in the above inequality is the only place where the proof of Lemma~\ref{lem:dtr} differs from the proof of ~\cite[Lemma 3.7]{Kap-Wil}.)

If $p$ is not in the cut locus of $q$
and $\gamma_{pq}\colon [0,1]\rightarrow M$
is a minimal geodesic between $p$ and $q$,
then by Lemma~\ref{lem:dist-growth}
\begin{align}
dt'_r(0)(p,q)\le d(p,q)+\int_0^1\|\nabla_\cdot X\|(\gamma_{pq}(t))\,dt\,.
\end{align}
Combining the last two inequalities with the segment inequality
we deduce

\begin{eqnarray}
h'(s)&\le& C_1(n,\lambda)r\fint_{B_{6r}(c(s))}\|\nabla_{\cdot}X\|\\
&\le& C_1(n,\lambda)r\Mx\!_1\|\nabla_{\cdot}X\|(c(s))\,.\notag
\end{eqnarray}
Note that the choice of the constant $C_1(n,\lambda)$ can be made explicit and  {\it independent}
of the induction assumption. We deduce $h(1)\le C_1(n,\lambda)r\mu$ and thus
the subset
\begin{align}
B_r(p_0)':= \biggr\{ p\in B_r(p_0)\Bigm| \fint_{B_{r/10}(p_0)'}\dt_r(1)(p,q)\,d\vol_q\le r/2\biggr\}
\end{align}
satisfies
\begin{eqnarray}\label{vol est1}
\vol(B_r(p_0)')\ge (1-2C_1(n,\lambda)\mu)\vol(B_r(p_0))\,.
\end{eqnarray}
It is elementary to check
that
\begin{align}
\phi_t(B_r(p_0)')\subset B_{2r}(c(t)) \mbox{ for all $t\in [0,1]$\,.}
\end{align}
Then arguing as before we estimate that
\begin{align}
\fint_{B_r(p_0)'\times B_r(p_0)}dt_r(1)(p,q)d\vol_p d\vol_q\le C_2(n,\lambda)\cdot r\cdot \mu\,.
\end{align}
Using $\dt_r(1)\le r$ and the volume estimate (\ref{vol est1}) this gives
\begin{align}
\fint_{B_r(c(p_0))^2}\dt_r(1)(p,q)\,d\vol_p\, d\vol_q\le C_2\cdot r\cdot \mu+2rC_1\mu=: C_3r\mu\,.
\end{align}
This completes the induction step with $C(n,\lambda)=C_3$ and $\mu_0=\tfrac{1}{2C_3}$.
In order to remove the restriction $\mu\le \mu_0$ one can just increase $C(n,\lambda)$ by the factor $4$, as indicated at the beginning.
\end{proof}

\begin{rmk}
It's obvious from the proof that Lemma~\ref{lem:dtr} and Corollary ~\ref{cor-good-pts} remain valid for $r\le \rho/10$ if we change $\Mx _1$ to $\Mx _\rho$ in the assumptions.
\end{rmk}


We can now finish the proof of Theorem~\ref{thm-main1} by establishing the following

\begin{lem}\label{lem-bilip}
Fix a small $R<r_0/10$ and let $\nu=\eta\ll 1$ and let  $\eps\le \frac{\nu^4}{C^2(n,\delta) }$  (where $C(n,\delta)$ is the constant in \eqref{eq:hess-bnd-most-pts1} ) satisfies Lemma~\ref{lem: distort2}.
Then for any $s\le\eps$ we have
\begin{enumerate}
\item \label{bilip-i} the map $\phi_s|_{\mathcal{B}^t_\eps(R,\nu)}$ is $(1+(C(n,\delta)\nu^{1/n})$-Bilipschitz onto its image,
\item \label{holder-incl-ii}
\begin{equation*}\label{eq:holder-incl}
\phi_s(\mathcal{B}^t_\eps(R,\nu))\subset B_{(1+C(n,\delta)\nu^{1/n})R}(\gamma(t-s))
\end{equation*}
for any $s\le\eps$.
\end{enumerate}
\end{lem}
\begin{proof}
Since  $\eps\le \frac{\nu^4}{C^2(n,\delta) }$ in  \eqref {eq:hess-bnd-most-pts1} we have  $ \frac {C(n,\delta) \sqrt\eps} {\nu}\le\nu$  and therefore by the definition of $\mathcal{B}^t_\eps(R,\nu)$ for all $x\in \mathcal{B}^t_\eps(R,\nu)$ it holds
\begin{equation}
\int_0^\eps \MxR |\Hess_h(\phi_s(x))|\,ds\le\nu\,.
\end{equation} This means that we can apply Corollary \ref{cor-good-pts} at all such points with $\lambda=\lambda(n,\delta)$ and $\mu=\nu$.
Let $C_1(n,\lambda(n,\delta))=C_1(n,\delta)$ be the constant provided by Corollary \ref{cor-good-pts}.

Let $x,y\in \mathcal{B}^t_\eps(R,\nu)$.  Let $r_1=C_2(n,\delta)\nu^{1/n} d(x,y)$ and set $r=0.5d(x,y)+r_1$. Using Corollary \ref{cor-good-pts}, by Bishop-Gromov we can be assured that $B_r(x)''\cap B_r(y)''\ne \varnothing$ provided $C_2(n,\delta)\gg C_1(n,\delta)$  and $\nu$ is chosen  small enough. Pick any $z\in B_r(x)''\cap B_r(y)''\ne \varnothing$.  Likewise, using  Bishop-Gromov we can find  $x_1\in B_r(x)''\cap B_{r_1}(x)$ and  $y_1\in B_r(y)''\cap B_{r_1}(y)$.

Therefore, by  Corollary \ref{cor-good-pts}  we have
\begin{align}
 &d(\phi_s(x), \phi_s(x_1))\le 2d(x,x_1)\le   C(n,\delta)\nu^{1/n}d(x,y),
 \\
 &d(\phi_s(y), \phi_s(y_1))\le 2d(y,y_1)\le C(n,\delta)\nu^{1/n}d(x,y),
 \\
 &d(\phi_s(x_1),\phi_s(z))\le d(x_1,z)+C(n,\delta)\nu r\le (0.5+C(n,\delta)\nu^{1/n})d(x,y)
\end{align}
and
\begin{align}
 d(\phi_s(y_1),\phi_s(z))\le d(y_1,z)+C(n,\delta)\nu r\le (0.5+C(n,\delta)\nu^{1/n})d(x,y).
\end{align}
Summing up the above inequalities and using the triangle inequality we get
\begin{align}
  d(\phi_s(x),\phi_t(y))
  &\le d(\phi_s(x),\phi_s(x_1))+d(\phi_s(x_1),\phi_s(z)) +d(\phi_s(z),\phi_s(y_1))+d(\phi_s(y_1),\phi_t(y))
  \notag \\
  &\le C(n,\delta)\nu^{1/n}d(x,y)+(0.5+C(n,\delta)\nu^{1/n})d(x,y)
  \notag\\ 
  &\qquad+(0.5+C(n,\delta)\nu^{1/n})d(x,y)+C(n,\delta)\nu^{1/n}d(x,y)
  \notag \\
  &\le (1+C(n,\delta))\nu^{1/n})d(x,y).
\end{align}
  This shows that $\phi_s$ is $(1+C(n,\delta))\nu^{1/n})$-Lipschitz on  $\mathcal{B}^t_\eps(R,\nu)$.

Next, let us show  that it's Bilipschitz.
As before, using Bishop-Gromov, we can find $y_2\in B_{d(x,y)-r_1}(x)''\cap B_{2r_1}(y)''$ and $x_2\in  B_{r_1}(x)''\cap  B_{d(x,y)-r_1}(x)''$. This implies that
\begin{align}
  &d(\phi_s(y_2),\phi_s(y))\le 4r_1\le C(n,\delta)\nu^{1/n} d(x,y),
  \\
  &d(\phi_s(x_2),\phi_s(x))\le 2r_1\le C(n,\delta)\nu^{1/n} d(x,y)
\end{align}
and
\begin{align}
  d(\phi_s(x_2),\phi_s(y_2))) \ge d(x_2,y_2) -C(n,\delta)\nu d(x,y) \ge (1-C(n,\delta)\nu^{1/n})d(x,y).
\end{align}
By the triangle inequality this yields
\begin{align}
  d(\phi_s(x),\phi_s(y)))
  &\ge (1-C(n,\delta)\nu^{1/n})d(x,y)-C(n,\delta))\nu^{1/n} d(x,y)-C(n,\delta)\nu^{1/n}d(x,y)
  \notag \\
  &\ge (1-C(n,\delta)\nu^{1/n})d(x,y),
\end{align}
which finally proves part \eqref{bilip-i}  of  Lemma~\ref{lem-bilip}.

Let us prove part \eqref{holder-incl-ii}.

Recall that by Lemma~\ref{lem: distort2}  and \eqref{eq:hess-bnd-most-pts2} we have that
\begin{align}
 \frac{\vol\mathcal{B}^t_\eps(R,\nu)}{\Vol(B_R(\gamma(t)))} \geq 1-C(n,\delta)\,\nu\, .
\end{align}

Therefore, By Bishop-Gromov we can find $z\in  \mathcal{B}^t_\eps(R,\nu)\cap  \mathcal{B}^t_\eps((C(n,\delta)\nu^{1/n}R,\nu)$. Since by construction we have that  $\mathcal{B}^t_\eps((C(n,\delta)\nu^{1/n}R,\nu)\subset  \mathcal{B}^t_\eps((C(n,\delta)\nu^{1/n}R)$, by Lemma ~\ref{lem: distort2} we have that $d(\phi_s(z),\gamma(t-s))\le 2 C(n,\delta)\nu^{1/n}R$. By part  \eqref{bilip-i},   for any $x\in  \mathcal{B}^t_\eps(R,\nu)$ we also have that $d(\phi_s(z),\phi_s(x))\le (1+(C(n,\delta)\nu^{1/n})d(x,z)\le (1+(C(n,\delta)\nu^{1/n})R$. Applying the triangle inequality we get that $d(\phi_s(x),\gamma(t-s))\le (1+(C(n,\delta)\nu^{1/n})R$. This yields \eqref{eq:holder-incl} and  finishes the proof of Lemma~\ref{lem-bilip} and hence of Theorem~\ref{thm-main1}.

\end{proof}

\small
\bibliographystyle{alpha}

\end{document}